\documentclass[final,3p]{elsarticle}
 \usepackage{graphics}
 \usepackage{graphicx}
 \usepackage{epsfig}
\usepackage{amssymb}
 \usepackage{amsthm}
 \usepackage{lineno}
 \usepackage{amsmath}
   \numberwithin{equation}{section}
\usepackage{mathrsfs}

\newtheorem{thm}{Theorem}[section]

\newtheorem{lem}[thm]{Lemma}

\newtheorem{defn}[thm]{Definition}

\biboptions{sort&compress,square}
\begin{document}
\begin{frontmatter}

\title{Nonminimal De Rham-Hodge Operators and Non-commutative Residue }
\begin{abstract}

In this paper, we get a Kastler-Kalau-Walze type theorem associated to nonminimal de Rham-Hodge operators on compact manifolds with  boundary.
We give two kinds of operator-theoretic explanations of the gravitational action in the case
of four dimensional compact manifolds with flat boundary.
\end{abstract}
\end{frontmatter}
\section{Introduction}
\label{1}
The noncommutative residue found in \cite{Gu,Wo1} plays a prominent role in noncommutative geometry.
For one dimensional manifolds, the noncommutative residue was discovered by Adler \cite{MA}
 in connection with geometric aspects of nonlinear partial differential equations.
In \cite{Co1}, Connes used the noncommutative residue to derive a conformal four dimensional Polyakov action analogy. Moreover, in \cite{Co2},
Connes made a challenging observation that the noncommutative residue of the square of the inverse of the Dirac operator was proportional to
the Einstein-Hilbert action, which we call the Kastler-Kalau-Walze theorem.
In \cite{Ka}, Kastler gave a brute-force proof of this theorem. Then Kalau and Walze also gave a proof of this theorem by using normal coordinates
 \cite{KW}.

In \cite{FGLS}, Fedosov etc. defined a noncommutative residue on Boutet de Monvel's algebra and proved that it was a
unique continuous trace. In \cite{S}, Schrohe gave the relation between the Dixmier trace and the noncommutative residue for
manifolds with boundary.  For an oriented spin manifold  $M$ with boundary $\partial M$,  by the composition formula in Boutet de Monvel's algebra and
the definition of $ \widetilde{{\rm Wres}}$ \cite{Wa1},  $\widetilde{{\rm Wres}}[(\pi^+D^{-1})^2]$ should be the sum of two terms from
interior and boundary of $M$, where $\pi^+D^{-1}$ is an element in Boutet de Monvel's algebra  \cite{Wa1}.
 For lower-dimension spin manifolds with boundary and the associated Dirac operators, Wang computed the lower dimensional volume and got
 a Kastler-Kalau-Walze type theorem in \cite{Wa4}, \cite{Wa3} and \cite{Wa2}. In\cite{GBF}, Gilkey, Branson and Fulling obtained a formula about
  heat kernel expansion coefficients of nonminimal operators. In \cite{WW}, we considered  the non-commutative residue of nonminimal operators
   and got the Kastler-Kalau-Walze type theorems for  nonminimal operators.
 The motivation of this paper is to generalize Theorem 2.2 in \cite{WW} to manifolds with  boundary in the four dimensional case.
For this purpose, we introduce the nonminimal de Rham-Hodge operators  $\tilde{D}=a d+b\delta $
 and nonminimal  laplacian  operators  $\tilde{D}\tilde{D^{*}}=a^{2} d\delta+b^{2}\delta d $. Our main result is as follows:

\textbf{Theorem 3.5} The following equallity holds:
\begin{equation*}
\widetilde{{\rm
Wres}}[\pi^+\tilde{D}^{-1} \circ \pi^+(\tilde{D^{*}})^{-1} ]
=4\pi\int_{M}\sum _{k=0} ^{4}c_{1}(4,k,a,b)R\texttt{d}vol(M)
-\frac{23}{12} (\frac{1}{a^{2}}+\frac{1}{b^{2}} )\pi\int_{\partial M} K  \Omega_{3}\texttt{d}x',
\end{equation*}
(for related definition, see Section 3).

  This paper is organized as follows: In Section 2, we define lower dimensional volumes of manifolds with boundary for nonminimal
  De-Rham Hodge operators. In Section 3, for four dimensional compact manifolds with boundary and  the associated nonminimal De-Rham Hodge operators
  $\tilde{D}=a d+b\delta $ and $\tilde{D^{*}}=b d+a\delta $,
 we compute the lower dimensional volume ${\rm Vol}^{(1,1)}_4$ and get a Kastler-Kalau-Walze type theorem in this case.
 In Section 4, for four dimensional compact manifolds with boundary and the associated nonminimal
De-Rham Hodge operators $\tilde{D}=a d+b\delta $, we compute the lower dimensional volume.
 When $\partial M$ is flat, we give two kinds of operator theoretic explanations of
the gravitational action on boundary.

\section{Lower dimensional volumes of Riemannian manifolds with boundary }
 In this section we consider an $n$-dimensional oriented Riemannian manifold $(M, g^{M})$ equipped
with some spin structure. Let $M$ be an $n$-dimensional compact oriented manifold with boundary $\partial M$.
 We assume that the metric $g^{M}$ on $M$ has
the following form near the boundary
 \begin{equation*}
 g^{M}=\frac{1}{h(x_{n})}g^{\partial M}+\texttt{d}x _{n}^{2} ,
\end{equation*}
where $g^{\partial M}$ is the metric on $\partial M$. Let $U\subset
M$ be a collar neighborhood of $\partial M$ which is diffeomorphic $\partial M\times [0,1)$. By the definition of $h(x_n)\in C^{\infty}([0,1))$
and $h(x_n)>0$, there exists $\tilde{h}\in C^{\infty}((-\varepsilon,1))$ such that $\tilde{h}|_{[0,1)}=h$ and $\tilde{h}>0$ for some
sufficiently small $\varepsilon>0$. Then there exists a metric $\hat{g}$ on $\hat{M}=M\bigcup_{\partial M}\partial M\times
(-\varepsilon,0]$ which has the form on $U\bigcup_{\partial M}\partial M\times (-\varepsilon,0 ]$
 \begin{equation*}
\hat{g}=\frac{1}{\tilde{h}(x_{n})}g^{\partial M}+\texttt{d}x _{n}^{2} ,
\end{equation*}
such that $\hat{g}|_{M}=g$. We fix a metric $\hat{g}$ on the $\hat{M}$ such that $\hat{g}|_{M}=g$.

Let $\nabla$ denote the Levi-civita connection about $g^M$.
 In the local coordinates $\{x_i; 1\leq i\leq n\}$ and the fixed orthonormal frame $\{\widetilde{e_1},\cdots,\widetilde{e_n}\}$,
 the connection matrix $(\omega_{s,t})$ is defined by
  \begin{equation*}
\nabla(\widetilde{e_1},\cdots,\widetilde{e_n})= (\widetilde{e_1},\cdots,\widetilde{e_n})(\omega_{s,t}).
\end{equation*}
Let $c( e_i)=\epsilon(e_i)-\iota(e_i),~~\hat{c}(\widetilde{e_i})=\epsilon(e_i)+\iota(e_i)$ .
Denote the exterior and interior multiplications by $\epsilon (e_j* ),~\iota (e_j* )$ respectively, denote
by  $d+\delta:~\wedge^*(T^*M)\rightarrow \wedge^*(T^*M)$  the signature operator.
By \cite{Y}, we have
  \begin{equation*}
\texttt{d}+\delta=\sum^n_{i=1}c(e_i)\Big[e_i+\frac{1}{4}\sum_{s,t}\omega_{s,t}
(e_i)\big[\hat{c}(e_s)\hat{c}(e_t)-c(e_s)c(e_t)\big]\Big].
\end{equation*}
 Then we define the nonminimal de Rham-Hodge operators  as
   \begin{equation*}
\tilde{D}=a\texttt{d}+b\delta,~~\tilde{D^{*}}=b\texttt{d}+a\delta,
\end{equation*}
where $\tilde{D^{*}}$ is the adjoint operator of $\tilde{D}$ and $ab\neq 0$.

To define the lower dimensional volume, some basic facts and formulae about Boutet de Monvel's calculus which can be found  in Sec.2
in \cite{Wa1} and \cite{BDM}
are needed. Let
\begin{equation*}
  F:L^2({\bf R}_t)\rightarrow L^2({\bf R}_v);~F(u)(v)=\int e^{-ivt}u(t)dt
  \end{equation*}
   denote the Fourier transformation and
$\Phi(\overline{{\bf R}^+}) =r^+\Phi({\bf R})$ (similarly define $\Phi(\overline{{\bf R}^-}$)), where $\Phi({\bf R})$
denotes the Schwartz space and
  \begin{equation*}
r^{+}:C^\infty ({\bf R})\rightarrow C^\infty (\overline{{\bf R}^+});~ f\rightarrow f|\overline{{\bf R}^+};~
 \overline{{\bf R}^+}=\{x\geq0;x\in {\bf R}\}.
\end{equation*}
We define $H^+=F(\Phi(\overline{{\bf R}^+}));~ H^-_0=F(\Phi(\overline{{\bf R}^-}))$ which are orthogonal to each other. We have the following
 property: $h\in H^+~(H^-_0)$ iff $h\in C^\infty({\bf R})$, which has an analytic extension to the lower (upper) complex
half-plane $\{{\rm Im}\xi<0\}~(\{{\rm Im}\xi>0\})$ such that for all nonnegative integer $l$,
 \begin{equation*}
\frac{d^{l}h}{d\xi^l}(\xi)\sim\sum^{\infty}_{k=1}\frac{d^l}{d\xi^l}(\frac{c_k}{\xi^k})
\end{equation*}
as $|\xi|\rightarrow +\infty,{\rm Im}\xi\leq0~({\rm Im}\xi\geq0)$.

 Let $H'$ be the space of all polynomials and $H^-=H^-_0\bigoplus H';~H=H^+\bigoplus H^-.$ Denote by $\pi^+~(\pi^-)$ respectively the
 projection on $H^+~(H^-)$. For calculations, we take $H=\widetilde H=\{$rational functions having no poles on the real axis$\}$ ($\tilde{H}$
 is a dense set in the topology of $H$). Then on $\tilde{H}$,
 \begin{equation}
\pi^+h(\xi_0)=\frac{1}{2\pi i}\lim_{u\rightarrow 0^{-}}\int_{\Gamma^+}\frac{h(\xi)}{\xi_0+iu-\xi}\texttt{d}\xi,
\end{equation}
where $\Gamma^+$ is a Jordan close curve included ${\rm Im}\xi>0$ surrounding all the singularities of $h$ in the upper half-plane and
$\xi_0\in {\bf R}$. Similarly, define $\pi^{'}$ on $\tilde{H}$,
 \begin{equation*}
\pi'h=\frac{1}{2\pi}\int_{\Gamma^+}h(\xi)\texttt{d}\xi.
\end{equation*}
So, $\pi'(H^-)=0$. For $h\in H\bigcap L^1(R)$, $\pi'h=\frac{1}{2\pi}\int_{R}h(v)dv$ and for $h\in H^+\bigcap L^1(R)$, $\pi'h=0$.

Denote by $\mathcal{B}$ Boutet de Monvel's algebra, we recall the main theorem in \cite{FGLS}.
\begin{thm}\label{th:32}{\bf(Fedosov-Golse-Leichtnam-Schrohe)}
 Let $X$ and $\partial X$ be connected, ${\rm dim}X=n\geq3$,
 $A=\left(\begin{array}{lcr}\pi^+P+G &   K \\
T &  S    \end{array}\right)$ $\in \mathcal{B}$ , and denote by $p$, $b$ and $s$ the local symbols of $P,G$ and $S$ respectively.
 Define:
 \begin{eqnarray}
{\rm{\widetilde{Wres}}}(A)&=&\int_X\int_{\bf S}{\rm{tr}}_E\left[p_{-n}(x,\xi)\right]\sigma(\xi)dx \nonumber\\
&&+2\pi\int_ {\partial X}\int_{\bf S'}\left\{{\rm tr}_E\left[({\rm{tr}}b_{-n})(x',\xi')\right]+{\rm{tr}}
_F\left[s_{1-n}(x',\xi')\right]\right\}\sigma(\xi')dx'.
\end{eqnarray}
Then~~ a) ${\rm \widetilde{Wres}}([A,B])=0 $, for any
$A,B\in\mathcal{B}$;~~ b) It is a unique continuous trace on
$\mathcal{B}/\mathcal{B}^{-\infty}$.
\end{thm}
Let $p_{1},p_{2}$ be nonnegative integers and $p_{1}+p_{2}\leq n$. Then by Sec 2.1 of \cite{Wa3},  we have
\begin{defn} Lower-dimensional volumes of Riemannian manifolds with boundary  are defined by
   \begin{equation}\label{}
  {\rm Vol}^{(p_1,p_2)}_nM:=\widetilde{{\rm Wres}}[\pi^+\tilde{D}^{-p_1}\circ\pi^+(\tilde{D^{*}})^{-p_2}].
\end{equation}
\end{defn}

 Denote by $\sigma_{l}(A)$ the $l$-order symbol of an operator A. For $n$ dimensional Riemannian manifolds with boundary,
 an application of (2.1.4) in \cite{Wa1} shows that
\begin{equation}
\widetilde{{\rm Wres}}[\pi^+\tilde{D}^{-p_1}\circ\pi^+(\tilde{D^{*}})^{-p_2}]=\int_M\int_{|\xi|=1}{\rm
trace}_{\Lambda^{*}(T^{*}M)}[\sigma_{-n}(\tilde{D}^{-p_1}  (\tilde{D^{*}})^{-p_2})]\sigma(\xi)\texttt{d}x+\int_{\partial
M}\Phi,
\end{equation}
where
 \begin{eqnarray}
\Phi&=&\int_{|\xi'|=1}\int^{+\infty}_{-\infty}\sum^{\infty}_{j, k=0}
\sum\frac{(-i)^{|\alpha|+j+k+1}}{\alpha!(j+k+1)!}
 {\rm trace}_{\Lambda^{*}(T^{*}M)}
\Big[\partial^j_{x_n}\partial^\alpha_{\xi'}\partial^k_{\xi_n}
\sigma^+_{r}(\tilde{D}^{-p_1})(x',0,\xi',\xi_n)\nonumber\\
&&\times\partial^\alpha_{x'}\partial^{j+1}_{\xi_n}\partial^k_{x_n}\sigma_{l}
((\tilde{D^{*}})^{-p_2})(x',0,\xi',\xi_n)\Big]d\xi_n\sigma(\xi')\texttt{d}x',
\end{eqnarray}
and the sum is taken over $r-k+|\alpha|+\ell-j-1=-n,r\leq-p_{1},\ell\leq-p_{2}$.

\section{A Kastler-Kalau-Walze type theorem of nonminimal de Rham-Hodge operators $\tilde{D}$ and $\tilde{D^{*}}$}
 \label{4}
In this section, we compute the lower dimension volume for four dimension compact connected manifolds with boundary
associated to nonminimal de Rham-Hodge operators $\tilde{D}$ and $\tilde{D^{*}}$
and get a Kastler-Kalau-Walze type formula in this case.
Let $M$ be an four dimensional compact oriented connected manifold with boundary $\partial M$, and the metric $g^{M}$ on $M$ as above.
Note that
   \begin{equation*}
\tilde{D}\tilde{D^{*}}=a^{2}\texttt{d}\delta+b^{2}\delta \texttt{d}
\end{equation*}
is a nonminimal operator on $C^{\infty}({\Lambda^{*}(T^{*}M)})$,
then $[\sigma_{-4}((\tilde{D}\tilde{D^{*}})^{-1})]|_{M}$ has the same expression with the case of  without boundary in \cite{WW},
so locally we can use Theorem 2.2 in \cite{WW} to compute the first term. Therefore
\begin{equation}
\int_{M}\int_{|\xi|=1}
  \text{trace}_{\Lambda^{*}(T^{*}M)}[\sigma_{-4}((\tilde{D}\tilde{D^{*}})^{-1})]\sigma(\xi)\text{d}x
=4\pi\int_{M}\sum _{k=0} ^{4}c_{1}(4,k,a,b)R\texttt{d}vol(M),
\end{equation}
where $R$ is the scalar curvature and
$c_{1}(4,k,a,b)=b^{-2}\{\frac{1}{6}(_{k}^{4})-(_{k-1}^{2})\}
+(b^{-2}-a^{-2})\sum_{j<k}(-1)^{j-k}\{\frac{1}{6}(_{j}^{4})-(_{j-1}^{2})\}.$

Hence we only need to compute $\int_{\partial M}\Phi$.
Firstly, we compute the symbol $\sigma(\tilde{D}^{-1})$ and $\sigma((\tilde{D^{*}})^{-1})$.
Denote by  $\tilde{c}(\xi)=a\epsilon(\xi)-b\iota(\xi)$, $\bar{c}(\xi)=b\epsilon(\xi)-a\iota(\xi) $,
 then we have $\tilde{c}(e_j*)=a\epsilon(e_j*)-b\iota(e_j*)$ and $\bar{c}(e_j*)=b\epsilon(e_j*)-a\iota(e_j*) $.
From the form of Signature operator
\begin{equation}
\tilde{D}=\sum^n_{i=1}\tilde{c}(e_i)\Big[ e_i+\sum_{s,t}\omega_{s,t}
(e_i)\Big(\hat{c}(e_s)\hat{c}(e_t)-c(e_s)c(e_t)\Big)\Big],
\end{equation}
we get
\begin{eqnarray}
\sigma_1(\tilde{D})&=&\sqrt{-1}\tilde{c}(\xi); \\
 \sigma_0(\tilde{D})&=&\frac{1}{4}\sum_{i,s,t}\omega_{s,t}
(e_i)\tilde{c}(e_i)[\hat{c}(e_s)\hat{c}(e_t)-c(e_s)c(e_t)],
\end{eqnarray}
where $\xi=\sum^n_{i=1}\xi_idx_i$ denotes the cotangent vector.
Write
\begin{equation}
\tilde{D}_x^{\alpha}=(-\sqrt{-1})^{|\alpha|}\partial_x^{\alpha};~\sigma(\tilde{D})=p_1+p_0;
~\sigma(\tilde{D}^{-1})=\sum^{\infty}_{j=1}q_{-j}.
\end{equation}
By the composition formula of psudodifferential operators,  we have
\begin{eqnarray*}
1=\sigma(\tilde{D}\circ \tilde{D}^{-1})&=&\sum_{\alpha}\frac{1}{\alpha!}\partial^{\alpha}_{\xi}[\sigma(\tilde{D})]D^{\alpha}_{x}[\sigma(\tilde{D}^{-1})]\\
&=&(p_1+p_0)(q_{-1}+q_{-2}+q_{-3}+\cdots)\\
& &~~~+\sum_j(\partial_{\xi_j}p_1+\partial_{\xi_j}p_0)(
D_{x_j}q_{-1}+D_{x_j}q_{-2}+D_{x_j}q_{-3}+\cdots)\\
&=&p_1q_{-1}+(p_1q_{-2}+p_0q_{-1}+\sum_j\partial_{\xi_j}p_1D_{x_j}q_{-1})+\cdots.
\end{eqnarray*}
Thus, we obtain
\begin{equation}
q_{-1}=p_1^{-1};~q_{-2}=-p_1^{-1}[p_0p_1^{-1}+\sum_j\partial_{\xi_j}p_1D_{x_j}(p_1^{-1})].
\end{equation}
By (3.2), (3.5) and direct computations, we have

\begin{lem} Let $\tilde{D}$, $\tilde{D^{*}}$ on $C^{\infty}({\Lambda^{*}(T^{*}M)})$. Then
\begin{eqnarray}
&&q_{-1}(\tilde{D^{-1}})=\frac{\sqrt{-1}\tilde{c}(\xi)}{ab|\xi|^2};\\
&&q_{-2}(\tilde{D^{-1}})=\frac{\tilde{c}(\xi) \sigma_0(\tilde{D})\tilde{c}(\xi)}{a^{2}b^{2}|\xi|^4}
+\frac{\tilde{c}(\xi)}{a^{2}b^{2}|\xi|^6}\sum_j\tilde{c}(dx_j)\Big[\partial_{x_j}[\tilde{c}(\xi)]|\xi|^2-\tilde{c}(\xi)\partial_{x_j}(|\xi|^2)\Big]; \\
&&q_{-1}((\tilde{D^{*}})^{-1}))=\frac{\sqrt{-1}\bar{c}(\xi)}{ab|\xi|^2};\\
&&q_{-2}((\tilde{D^{*}})^{-1}))=\frac{\bar{c}(\xi) \sigma_0(\tilde{D^{*}})\bar{c}(\xi)}{a^{2}b^{2}|\xi|^4}
+\frac{\bar{c}(\xi)}{a^{2}b^{2}|\xi|^6}\sum_j\bar{c}(dx_j)\Big[\partial_{x_j}[\bar{c}(\xi)]|\xi|^2-\bar{c}(\xi)\partial_{x_j}(|\xi|^2)\Big],
\end{eqnarray}
where
\begin{eqnarray}
\tilde{ p_{0}}=\sigma_0(\tilde{D})(x_0)&=&-\frac{1}{4}h'(0)\sum^{n-1}_{i=1}\tilde{c}(e_i)\hat{c}(e_i)\hat{c}(e_n)(x_0)+
 \frac{1}{4}h'(0)\sum^{n-1}_{i=1}\tilde{c}(e_i)c(e_i)c(e_n)(x_0); \\
\bar{ p_{0}}=\sigma_0(\tilde{D^{*}})(x_0)&=&-\frac{1}{4}h'(0)\sum^{n-1}_{i=1}
\bar{c}(e_i)\hat{c}(e_i)\hat{c}(e_n)(x_0)+
 \frac{1}{4}h'(0)\sum^{n-1}_{i=1}\bar{c}(e_i)c(e_i)c(e_n)(x_0).
\end{eqnarray}
\end{lem}

Since $\Phi$ is a global form on $\partial M$, so for any fixed point $x_{0}\in\partial M$, we can choose the normal coordinates
$U$ of $x_{0}$ in $\partial M$(not in $M$) and compute $\Phi(x_{0})$ in the coordinates $\widetilde{U}=U\times [0,1)$ and the metric
$\frac{1}{h(x_{n})}g^{\partial M}+\texttt{d}x _{n}^{2}$. The dual metric of $g^{M}$ on $\widetilde{U}$ is
$h(x_{n})g^{\partial M}+\texttt{d}x _{n}^{2}.$ Write
$g_{ij}^{M}=g^{M}(\frac{\partial}{\partial x_{i}},\frac{\partial}{\partial x_{j}})$;
$g^{ij}_{M}=g^{M}(d x_{i},dx_{j})$, then

\begin{equation*}
[g_{i,j}^{M}]=
\begin{bmatrix}\frac{1}{h( x_{n})}[g_{i,j}^{\partial M}]&0\\0&1\end{bmatrix};\quad
[g^{i,j}_{M}]=\begin{bmatrix} h( x_{n})[g^{i,j}_{\partial M}]&0\\0&1\end{bmatrix},
\end{equation*}
and
\begin{equation}
\partial_{x_{s}} g_{ij}^{\partial M}(x_{0})=0,\quad 1\leq i,j\leq n-1;\quad g_{i,j}^{M}(x_{0})=\delta_{ij}.
\end{equation}
By Lemma 2.2 in \cite{Wa3}  and  the normal coordinates $U$ of $x_{0}$ in $\partial M$(not in $M$), we have
\begin{lem}\label{le:32}
With the metric $g^{M}$ on $M$ near the boundary
\begin{eqnarray}
\partial_{x_j}(|\xi|_{g^M}^2)(x_0)&=&\left\{
       \begin{array}{c}
        0,  ~~~~~~~~~~ ~~~~~~~~~~ ~~~~~~~~~~~~~{\rm if }~j<n; \\[2pt]
       h'(0)|\xi'|^{2}_{g^{\partial M}},~~~~~~~~~~~~~~~~~~~~{\rm if }~j=n,
       \end{array}
    \right. \\
\partial_{x_j}[\tilde{c}(\xi)](x_0)&=&\left\{
       \begin{array}{c}
      0,  ~~~~~~~~~~ ~~~~~~~~~~ ~~~~~~~~~~~~~{\rm if }~j<n;\\[2pt]
\partial x_{n}(\tilde{c}(\xi'))(x_{0}), ~~~~~~~~~~~~~~~~~{\rm if }~j=n,
       \end{array}
    \right.
\end{eqnarray}
where $\xi=\xi'+\xi_{n}\texttt{d}x_{n}$.
\end{lem}

\begin{lem} \cite{Wa3}
 When $ i<n,~\omega_{n,i}(\widetilde{e_i})(x_0)=\frac{1}{2}h'(0);$
and $\omega_{i,n}(\widetilde{e_i})(x_0)=-\frac{1}{2}h'(0),$
{\it In other cases,} $\omega_{s,t}(\widetilde{e_i})(x_0)=0$.
\end{lem}

\begin{lem}
By the relation of the Clifford action and ${\rm tr}{AB}={\rm tr }{BA}$, then we have the equalities:
\begin{eqnarray}
&&{\rm tr}[\tilde{c}(\xi')\tilde{p}_{0}\tilde{c}(\xi')\epsilon(dx_n)](x_0)|_{|\xi'|=1}=6ab^{2}h'(0);
{\rm tr}[\tilde{c}(\xi')\tilde{p}_{0}\tilde{c}(\xi')\iota(dx_n)](x_0)|_{|\xi'|=1}=-6a^{2}b h'(0);\nonumber\\
&&{\rm tr}[\tilde{c}(dx_n)\tilde{p}_{0}\tilde{c}(dx_n)\epsilon(dx_n)](x_0)|_{|\xi'|=1}=-6b^{2}ah'(0);
{\rm tr}[\tilde{c}(dx_n)\tilde{p}_{0}\tilde{c}(dx_n)\iota(dx_n)](x_0)|_{|\xi'|=1}=6a^{2}b h'(0);\nonumber\\
&&{\rm tr}[\tilde{c}(dx_n)\tilde{p}_{0}\tilde{c}(\xi')\epsilon(\xi')](x_0)|_{|\xi'|=1}=-2b^{2}ah'(0);
{\rm tr}[\tilde{c}(dx_n)\tilde{p}_{0}\tilde{c}(\xi')\iota(\xi')](x_0)|_{|\xi'|=1}=10a^{2}b h'(0);\nonumber\\
&&{\rm tr}[\tilde{c}(\xi')\tilde{p}_{0}\tilde{c}(dx_n)\epsilon(\xi')](x_0)|_{|\xi'|=1}=-10b^{2}ah'(0);
{\rm tr}[\tilde{c}(\xi')\tilde{p}_{0}\tilde{c}(dx_n)\iota(\xi')](x_0)|_{|\xi'|=1}=2a^{2}b h'(0),
\end{eqnarray}
others vanishes.
\end{lem}
\begin{proof}
By the relation of the Clifford action and ${\rm tr}{AB}={\rm tr }{BA}$, then
\begin{eqnarray}
&&{\rm tr}[\tilde{c}(\xi')\tilde{p}_{0}\tilde{c}(\xi')\epsilon(dx_n)](x_0)|_{|\xi'|=1}\nonumber\\
&=&{\rm tr}[(a\epsilon(\xi')-b\iota(\xi'))\tilde{p}_{0}(a\epsilon(\xi')-b\iota(\xi'))\epsilon(dx_n)](x_0)|_{|\xi'|=1}\nonumber\\
&=&ab{\rm tr}[\epsilon(\xi')\iota(\xi')\tilde{p}_{0}\epsilon(dx_n)](x_0)|_{|\xi'|=1}
 +ab{\rm tr}[\iota(\xi')\epsilon(\xi')\tilde{p}_{0}\epsilon(dx_n)](x_0)|_{|\xi'|=1}\nonumber\\
 &=&ab{\rm tr}[\tilde{p}_{0}\epsilon(dx_n)](x_0)|_{|\xi'|=1},
\end{eqnarray}
where
\begin{eqnarray}
&&{\rm tr}[\tilde{p}_{0}\epsilon(dx_n)]\nonumber\\
&=&{\rm tr}[(-\frac{1}{4}h'(0)\sum^{n-1}_{i=1}\tilde{c}(e_i)\hat{c}(e_i)\hat{c}(e_n)+
 \frac{1}{4}h'(0)\sum^{n-1}_{i=1}\tilde{c}(e_i)c(e_i)c(e_n))\epsilon(dx_n)]\nonumber\\
&=&-\frac{1}{4}h'(0)\sum^{n-1}_{i=1}{\rm tr}[\epsilon(dx_n)\tilde{c}(e_i)\hat{c}(e_i)\hat{c}(e_n)]
+\frac{1}{4}h'(0)\sum^{n-1}_{i=1}{\rm tr}[\epsilon(dx_n)\tilde{c}(e_i)c(e_i)c(e_n)].
\end{eqnarray}
By the relation of the Clifford action and $\epsilon(e_{i})\iota(e_{j})+\iota(e_{j})\epsilon(e_{i})=\delta_{ij}$,
we obtain
\begin{eqnarray}
\sum^{n-1}_{i=1}{\rm tr}[\epsilon(dx_n)\tilde{c}(e_i)\hat{c}(e_i)\hat{c}(e_n)]
&=&a\sum^{n-1}_{i=1}{\rm tr}[\epsilon(e_{i})\iota(e_{i})\iota(e_{n})\epsilon(e_{n})]
-b\sum^{n-1}_{i=1}{\rm tr}[\iota(e_{i})\epsilon(e_{i})\iota(e_{n})\epsilon(e_{n})]\nonumber\\
&=&\frac{a}{2}\sum^{n-1}_{i=1}{\rm tr}[\epsilon(e_{i})\iota(e_{i})]
-b\sum^{n-1}_{i=1}{\rm tr}[\iota(e_{i})\epsilon(e_{i})]
=12(a-b),
\end{eqnarray}
and
\begin{equation}
\sum^{n-1}_{i=1}{\rm tr}[\epsilon(dx_n)\tilde{c}(e_i)c(e_i)c(e_n)]
=12(a+b).
\end{equation}
Combining (3.17)-(3.20), we have
\begin{equation}
{\rm tr}[\tilde{c}(\xi')\tilde{p}_{0}\tilde{c}(\xi')\epsilon(dx_n)](x_0)|_{|\xi'|=1}=6ab^{2}h'(0).
\end{equation}
Others are similarly.
\end{proof}

Let us now turn to compute $\Phi$ (see formula (2.5) for definition of $\Phi$). Since the sum is taken over $-r-\ell+k+j+|\alpha|=3,
 \ r, \ell\leq-1$, then we have the following five cases:

\textbf{Case a (I)}: \ $r=-1, \ \ell=-1, \ k=j=0, \ |\alpha|=1$

From (2.5) we have
\begin{equation}
\text{ Case a (\text{I}) }=-\int_{|\xi'|=1}\int_{-\infty}^{+\infty}\sum_{|\alpha|=1}\text{trace}
   \Big[\partial_{\xi'}^{\alpha}\pi_{\xi_{n}}^{+}\sigma_{-1}(\tilde{D}^{-1})
\times \partial_{x'}^{\alpha}\partial_{\xi_{n}}\sigma_{-1}((\tilde{D^{*}})^{-1})\Big](x_{0})\texttt{d}\xi_{n}
\sigma(\xi')\texttt{d}x'.
\end{equation}
Then an application of Lemma 3.2 shows that,
\begin{equation}
\partial_{x_i}\sigma_{-1}((\tilde{D^{*}})^{-1})(x_0)=\partial_{x_i}\left(\frac{\sqrt{-1}\bar{c}(\xi)}{ab|\xi|^2}\right)(x_0)=
\frac{\sqrt{-1}\partial_{x_i}[\bar{c}(\xi)](x_0)}{ab|\xi|^2}
-\frac{\sqrt{-1}\bar{c}(\xi)\partial_{x_i}(|\xi|^2)(x_0)}{ab|\xi|^4}=0,
 \end{equation}
so Case a (I) vanishes.

\textbf{Case a (II)}: \ $r=-1, \ \ell=-1, \ k=|\alpha|=0, \ j=1$

From (2.5) we have
\begin{equation}
\text{ Case a (\text{II}) }=-\frac{1}{2}\int_{|\xi'|=1}\int_{-\infty}^{+\infty}
             \text{trace}[\partial_{x_{n}}\pi_{\xi_{n}}^{+}\sigma_{-1}(\tilde{D}^{-1})
\times \partial_{\xi_{n}}^{2}\sigma_{-1}((\tilde{D^{*}})^{-1})](x_{0})\texttt{d}\xi_{n}\sigma(\xi')\texttt{d}x'.
\end{equation}
From Lemma 3.1 and Lemma 3.2, we have
\begin{equation}
\partial_{x_{n}}\sigma_{-1}(\tilde{D}^{-1})(x_{0})|_{|\xi'|=1}=\frac{\sqrt{-1}\partial_{x_n}[\tilde{c}(\xi)](x_0)}{ab|\xi|^2}
-\frac{\sqrt{-1}\tilde{c}(\xi)h'(0)}{ab|\xi|^4}.
\end{equation}
By the Cauchy integral formula, we obtain
\begin{equation}
\pi_{\xi_{n}}^{+}[\frac{1}{(1+\xi_{n}^{2})^{2}}](x_{0})|_{|\xi'|=1}
  =\frac{1}{2\pi i}\lim_{u\rightarrow 0^{-}}\int_{\Gamma^{+}}\frac{\frac{1}{(\eta_{n}+i)^{2}(\xi_{n}+iu-\eta_{n})}}
{(\eta_{n}-i)^{2}}\texttt{d}\eta_{n}=-\frac{i\xi_{n}+2}{4(\xi_{n}-i)^{2}},
\end{equation}
and
\begin{equation}
\pi^+_{\xi_n}\left[\frac{\sqrt{-1}\partial_{x_n}c(\xi')}{|\xi|^2}\right](x_0)|_{|\xi'|=1}
=\frac{\partial_{x_n}[c(\xi')](x_0)}{2(\xi_n-i)}.
\end{equation}
Then
\begin{eqnarray}
&&\partial_{x_{n}}\pi_{\xi_{n}}^{+}\sigma_{-1}(\tilde{D}^{-1})(x_{0})|_{|\xi'|=1}\nonumber\\
&=&\frac{\partial_{x_n}[\tilde{c}(\xi')](x_0)}{2ab(\xi_{n}-i)}
+\frac{ih'(0)}{ab}\Big[ \frac{i\tilde{c}(\xi')}{4(\xi_{n}-i)} +\frac{\tilde{c}(\xi')+i\tilde{c}(dx_{n})}{4(\xi_{n}-i)^{2}}\Big].
\nonumber\\
&=&-\frac{\partial_{x_n}[\iota(\xi')](x_0)}{2a(\xi_{n}-i)}
+\frac{ih'(0)}{ab}\Big[ \frac{a(i\xi_{n}+2)\epsilon(\xi')}{4(\xi_{n}-i)^{2}} -\frac{b(i\xi_{n}+2)\iota(\xi')}{4(\xi_{n}-i)^{2}}
+\frac{ai\epsilon(dx_{n})}{4(\xi_{n}-i)^{2}} -\frac{bi\iota(dx_{n})}{4(\xi_{n}-i)^{2}}\Big].
\end{eqnarray}
From  Lemma 3.2, we have
\begin{eqnarray}
 &&\partial_{\xi_{n}}^{2}\sigma_{-1}((\tilde{D^{*}})^{-1})(x_{0})|_{|\xi'|=1}\nonumber\\
&=&\frac{\sqrt{-1}}{ab}\left(-\frac{6\xi_n\bar{c} (dx_n)+2\bar{c} (\xi')}{|\xi|^4}+\frac{8\xi_n^2\bar{c} (\xi)}{|\xi|^6}\right)\nonumber\\
&=&\frac{\sqrt{-1}}{ab}\Big[ \frac{b(6\xi_{n}^{2}-2)\epsilon(\xi')}{(1+\xi_{n}^{2})^{3}} -\frac{a(6\xi_{n}^{2}-2)\iota(\xi')}{(1+\xi_{n}^{2})^{3}}
+\frac{b(2\xi_{n}^{3}-6\xi_{n})\epsilon(dx_{n})}{(1+\xi_{n}^{2})^{3}} -\frac{a(2\xi_{n}^{3}-6\xi_{n})\iota(dx_{n})}{(1+\xi_{n}^{2})^{3}}\Big].
\end{eqnarray}
By the relation of the Clifford action and ${\rm tr}{AB}={\rm tr }{BA}$, then we have the equalities:
\begin{eqnarray}
&&{\rm tr}[\epsilon(\xi')\iota(\xi')]=8;~~{\rm tr}[\epsilon(dx_n)\iota(dx_n)]=8;
{\rm tr}[\partial_{x_n}\iota(\xi')\epsilon(\xi')](x_0)|_{|\xi'|=1}=8h'(0);\nonumber\\
&&{\rm tr}[\partial_{x_n}\iota(\xi')\iota(dx_n)\epsilon(\xi')\epsilon(dx_n)](x_0)|_{|\xi'|=1}=-4h'(0);
{\rm tr}[\partial_{x_n}\iota(\xi')\iota(\xi')\epsilon(\xi')\epsilon(dx_n)](x_0)|_{|\xi'|=1}=0.
\end{eqnarray}
Combining (3.28), (3.29) and (3.30), we have
\begin{eqnarray}
 &&\text{trace}[\partial_{x_{n}}\pi_{\xi_{n}}^{+}\sigma_{-1}(\tilde{D}^{-1})
\times \partial_{\xi_{n}}^{2}\sigma_{-1}((\tilde{D^{*}})^{-1})](x_{0})|_{|\xi'|=1}\nonumber\\
&=&\frac{h'(0)}{a^{2}}\frac{8(-i\xi_{n}-i\xi_{n}^{3})}{(\xi_{n}-i)^{2}(1+\xi_{n}^{2})^{3}}
+\frac{h'(0)}{b^{2}}\frac{8(-1-2i\xi_{n}+3\xi_{n}^{2}+2i\xi_{n}^{3})}{(\xi_{n}-i)^{2}(1+\xi_{n}^{2})^{3}}.
\end{eqnarray}
Hence
\begin{eqnarray}
&&\text{ Case a (\text{II}) }\nonumber\\
&=&-\frac{1}{2}\int_{|\xi'|=1}\int_{-\infty}^{+\infty}
 \Big[\frac{h'(0)}{a^{2}}\frac{8(-i\xi_{n}-i\xi_{n}^{3})}{(\xi_{n}-i)^{2}(1+\xi_{n}^{2})^{3}}
+\frac{h'(0)}{b^{2}}\frac{8(-1-2i\xi_{n}+3\xi_{n}^{2}+2i\xi_{n}^{3})}{(\xi_{n}-i)^{2}(1+\xi_{n}^{2})^{3}}
 \Big]\texttt{d}\xi_{n}\sigma(\xi')\texttt{d}x'\nonumber\\
&=&-\frac{1}{2} \frac{h'(0)}{a^{2}}\Omega_{3}\int_{\Gamma^{+}}\frac{8(-i\xi_{n}-i\xi_{n}^{3})}{(\xi_{n}-i)^{2}(1+\xi_{n}^{2})^{3}}
\texttt{d}\xi_{n}\texttt{d}x'
-\frac{1}{2} \frac{h'(0)}{b^{2}}\Omega_{3}\int_{\Gamma^{+}}\frac{8(-1-2i\xi_{n}
+3\xi_{n}^{2}+2i\xi_{n}^{3})}{(\xi_{n}-i)^{2}(1+\xi_{n}^{2})^{3}}\texttt{d}\xi_{n}\texttt{d}x'\nonumber\\
&=&-\frac{1}{2} \frac{h'(0)}{a^{2}}\Omega_{3}\frac{2\pi i}{4!}
\Big[\frac{ 8(-i\xi_{n}-i\xi_{n}^{3}) }{(\xi_{n}+i)^{3}}\Big]^{(4)}|_{\xi_{n}=i}\texttt{d}x'
-\frac{1}{2} \frac{h'(0)}{b^{2}}\Omega_{3}\frac{2\pi i}{4!}
\Big[\frac{8(-1-2i\xi_{n}+3\xi_{n}^{2}+2i\xi_{n}^{3}) }{(\xi_{n}+i)^{3}}\Big]^{(4)}|_{\xi_{n}=i}\texttt{d}x' \nonumber\\
&=&-\frac{1}{2} \frac{h'(0)}{a^{2}}\Omega_{3}\frac{2\pi i}{4!}(-12i)\texttt{d}x'
-\frac{1}{2} \frac{h'(0)}{b^{2}}\Omega_{3}\frac{2\pi i}{4!}(-24i)\texttt{d}x' \nonumber\\
&=&(-\frac{1}{2a^{2}}  -\frac{1}{b^{2}})\pi h'(0)  \Omega_{3}\texttt{d}x',
\end{eqnarray}
where $\Omega_{3}$ is the canonical volume of $S^{3}.$

\textbf{Case a (III)}: \ $r=-1, \ \ell=-1, \ j=|\alpha|=0, \ k=1$

 From (2.5) we have
 \begin{equation}
\text{ Case a (\text{III}) }=-\frac{1}{2}\int_{|\xi'|=1}\int_{-\infty}^{+\infty}\text{trace}\Big[\partial_{\xi_{n}}\pi_{\xi_{n}}^{+}
\sigma_{-1}(\tilde{D}^{-1})\times\partial_{\xi_{n}}\partial_{x_{n}}\sigma_{-1}((\tilde{D^{*}})^{-1})\Big](x_{0})
    \texttt{d}\xi_{n}\sigma(\xi')\texttt{d}x'.
 \end{equation}
From Lemma 3.1 and Lemma 3.2, we get
\begin{eqnarray}
 \partial_{\xi_n}\pi^+_{\xi_n}\sigma_{-1}(\tilde{D}^{-1})(x_0)|_{|\xi'|=1}  &=&-\frac{\tilde{c}(\xi')+i\tilde{c}(dx_n)}{2ab(\xi_n-i)^2} \nonumber\\
      &=&-\frac{a\epsilon(\xi')-b\iota(\xi') +i(\epsilon(dx_n)+\iota(dx_n))}{2ab(\xi_n-i)^2},
\end{eqnarray}
and
\begin{eqnarray}
&&\partial_{\xi_n}\partial_{x_n}\sigma_{-1}((\tilde{D^{*}})^{-1}(x_0)|_{|\xi'|=1}\nonumber\\
&=&\frac{-\sqrt{-1}h'(0)}{ab}
\left[\frac{\bar{c}(dx_n)}{|\xi|^4}-4\xi_n\frac{\bar{c}(\xi')+\xi_n\bar{c}(dx_n)}{|\xi|^6}\right]-
\frac{2\xi_n\sqrt{-1}\partial_{x_n}\bar{c}(\xi')(x_0)}{ab|\xi|^4}\nonumber\\
&=&\frac{2i\xi_{n}\partial_{x_n}[\iota(\xi')](x_0)}{b|\xi|^4}
+\frac{ih'(0)}{ab}\Big[ \frac{4 b \xi_{n}\epsilon(\xi')}{|\xi|^6} -\frac{4a\xi_{n}\iota(\xi')}{|\xi|^6}
-\frac{b(|\xi|^2-4\xi_{n}^{2})\epsilon(dx_{n})}{|\xi|^6} +\frac{a(|\xi|^2-4\xi_{n}^{2})\iota(dx_{n})}{|\xi|^6}\Big].\nonumber\\
\end{eqnarray}
Combining (3.34) and (3.35), we obtain
\begin{eqnarray}
 &&\text{trace}\Big[\partial_{\xi_{n}}\pi_{\xi_{n}}^{+}
\sigma_{-1}(\tilde{D}^{-1})\times\partial_{\xi_{n}}\partial_{x_{n}}\sigma_{-1}((\tilde{D^{*}})^{-1})\Big](x_{0})|_{|\xi'|=1}\nonumber\\
&=&\frac{h'(0)}{a^{2}}\frac{4(1+4i\xi_{n}-3\xi_{n}^{3})}{(\xi_{n}-i)^{2}(1+\xi_{n}^{2})^{3}}
-\frac{h'(0)}{b^{2}}\frac{4(-1-2i\xi_{n}+3\xi_{n}^{2}+2i\xi_{n}^{3})}{(\xi_{n}-i)^{2}(1+\xi_{n}^{2})^{3}}.
\end{eqnarray}
Then
\begin{eqnarray}
&&\text{ Case a (\text{III}) }\nonumber\\
&=&-\frac{1}{2}\int_{|\xi'|=1}\int_{-\infty}^{+\infty}
 \Big[\frac{h'(0)}{a^{2}}\frac{4(1+4i\xi_{n}-3\xi_{n}^{3})}{(\xi_{n}-i)^{2}(1+\xi_{n}^{2})^{3}}
+\frac{h'(0)}{b^{2}}\frac{4(-1-2i\xi_{n}+3\xi_{n}^{2}+2i\xi_{n}^{3})}{(\xi_{n}-i)^{2}(1+\xi_{n}^{2})^{3}}
 \Big]\texttt{d}\xi_{n}\sigma(\xi')\texttt{d}x'\nonumber\\
&=&-\frac{1}{2} \frac{h'(0)}{a^{2}}\Omega_{3}\int_{\Gamma^{+}}\frac{4(1+4i\xi_{n}-3\xi_{n}^{3})}{(\xi_{n}-i)^{2}(1+\xi_{n}^{2})^{3}}
\texttt{d}\xi_{n}\texttt{d}x'
+\frac{1}{2} \frac{h'(0)}{b^{2}}\Omega_{3}\int_{\Gamma^{+}}\frac{4(-1-2i\xi_{n}+3\xi_{n}^{2}+2i\xi_{n}^{3})}
{(\xi_{n}-i)^{2}(1+\xi_{n}^{2})^{3}}\texttt{d}\xi_{n}\texttt{d}x'\nonumber\\
&=&-\frac{1}{2} \frac{h'(0)}{a^{2}}\Omega_{3}\frac{2\pi i}{4!}(-12i)\texttt{d}x'
+\frac{1}{2} \frac{h'(0)}{b^{2}}\Omega_{3}\frac{2\pi i}{4!}(24i)\texttt{d}x' \nonumber\\
&=&( \frac{1}{a^{2}} +\frac{1}{2b^{2}})\pi h'(0)  \Omega_{3}\texttt{d}x'.
\end{eqnarray}

\textbf{Case b}: \ $r=-2, \ \ell=-1, \ k=j=|\alpha|=0$

From (2.5) we have
\begin{equation}
\text{ Case b}=-i\int_{|\xi'|=1}\int_{-\infty}^{+\infty}\text{trace}[\pi_{\xi_{n}}^{+}\sigma_{-2}(\tilde{D}^{-1})
       \times\partial_{\xi_{n}}\sigma_{-1}((\tilde{D^{*}})^{-1})](x_{0})\texttt{d}\xi_{n}\sigma(\xi')\texttt{d}x' .
\end{equation}
 By Lemma 3.1 and Lemma 3.2, we obtain
 \begin{eqnarray}
&&\partial_{\xi_{n}}\sigma_{-1}((\tilde{D^{*}})^{-1})](x_{0})\nonumber\\
&=&\frac{\sqrt{-1}}{ab}\left(-\frac{2\xi_n^{2}\bar{c}(dx_n)+2\xi_n\bar{c}(\xi')}
{|\xi|^4}+\frac{\bar{c}(dx_n)}{|\xi|^2}\right)\nonumber\\
&=&\frac{\sqrt{-1}}{ab}\Big[ \frac{-2b\xi_{n}\epsilon(\xi')}{(1+\xi_{n}^{2})^{2}} +\frac{2a\xi_{n}\iota(\xi')}{(1+\xi_{n}^{2})^{2}}
+\frac{b(1-\xi_{n}^{2})\epsilon(dx_{n})}{(1+\xi_{n}^{2})^{2}} +\frac{a(\xi_{n}^{2}-1)\iota(dx_{n})}{(1+\xi_{n}^{2})^{2}}\Big]
\end{eqnarray}
 and
\begin{equation}
\sigma_{-2}(\tilde{D}^{-1})(x_0)=\frac{\tilde{c}(\xi)\tilde{p}_0(x_0)\tilde{c}(\xi)}{a^{2}b^{2}|\xi|^4}
+\frac{\tilde{c}(\xi)}{a^{2}b^{2}|\xi|^6}\tilde{c}(dx_n)
\Big[\partial_{x_n}[\tilde{c}(\xi')](x_0)|\xi|^2-\tilde{c}(\xi)h'(0)|\xi|^2_{\partial M}\Big].
\end{equation}
Then
 \begin{eqnarray}
&&\pi^+_{\xi_n}\sigma_{-2}(\tilde{D}^{-1})(x_0)|_{|\xi'|=1}\nonumber\\
&=&\pi^+_{\xi_n}\left[\frac{\tilde{c}(\xi)\tilde{p}_0(x_0)\tilde{c}(\xi)
+\tilde{c}(\xi)\tilde{c}(dx_n)\partial_{x_n}[\tilde{c}(\xi')](x_0)}{a^{2}b^{2}(1+\xi_n^2)^2}\right]
  -h'(0)\pi^+_{\xi_n}\left[\frac{\tilde{c}(\xi)\tilde{c}(dx_n)\tilde{c}(\xi)}{a^{2}b^{2}(1+\xi_n)^3}\right]\nonumber\\
&:=&B_1-B_2,
\end{eqnarray}
where
 \begin{eqnarray}
B_1&=&\frac{-1}{4a^{2}b^{2}(\xi_n-i)^2}\Big[(2+i\xi_n)\tilde{c}(\xi')\tilde{p}_0\tilde{c}(\xi')+i\xi_n\tilde{c}(dx_n)\tilde{p_0}\tilde{c}(dx_n)\nonumber\\
&&+(2+i\xi_n)\tilde{c}(\xi')\tilde{c}(dx_n)\partial_{x_n}\tilde{c}(\xi')+i\tilde{c}(dx_n)\tilde{p_0}\tilde{c}(\xi')
+i\tilde{c}(\xi')\tilde{p}_0\tilde{c}(dx_n)-i\partial_{x_n}\tilde{c}(\xi')\Big]\nonumber\\
&=&\frac{-1}{4a^{2}b^{2}(\xi_n-i)^2}\Big[(2+i\xi_n)\tilde{c}(\xi')\tilde{p}_0\tilde{c}(\xi')+i\xi_n\tilde{c}(dx_n)\tilde{p_0}\tilde{c}(dx_n)
+i\tilde{c}(dx_n)\tilde{p_0}\tilde{c}(\xi')+i\tilde{c}(\xi')\tilde{p}_0\tilde{c}(dx_n)\Big]\nonumber\\
&&+\frac{(2+i\xi_{n})\epsilon(\xi')\epsilon(dx_{n})\partial_{x_n}[\iota(\xi')](x_0)}{4b(\xi_{n}-i)^{2}}
 +\frac{b(2+i\xi_{n})\iota(\xi')\iota(dx_{n})\partial_{x_n}[\iota(\xi')](x_0)}{4a^{2}(\xi_{n}-i)^{2}}
 -\frac{i\partial_{x_n}[\iota(\xi')](x_0)}{4a(\xi_{n}-i)^{2}}.
\end{eqnarray}
From (3.39) and (3.42), we have
\begin{eqnarray}
 &&\text{trace}\Big[B_{1}\times\partial_{\xi_{n}}\sigma_{-1}((\tilde{D^{*}})^{-1})](x_{0})|_{|\xi'|=1}\nonumber\\
&=&\frac{h'(0)}{a^{2}}\times\frac{-i(3+5i\xi_{n}-3\xi_{n}^{2})}{(\xi_{n}-i)^{2}(1+\xi_{n}^{2})^{2}}
+\frac{h'(0)}{b^{2}}\times\frac{i(-1-23i\xi_{n}+\xi_{n}^{2}+2i\xi_{n}^{3})}{(\xi_{n}-i)^{2}(1+\xi_{n}^{2})^{2}}.
\end{eqnarray}
On the other hand,
\begin{eqnarray}
B_2&=&h'(0)\pi_{\xi_n}^+\left[\frac{-\xi_n^2c(dx_n)^2-2\xi_nc(\xi')+c(dx_n)}{a^{2}b^{2}(1+\xi_n^2)^3}\right]\nonumber\\
&=&\frac{h'(0)}{2ab}\left[\frac{\tilde{c}(dx_n)}{4i(\xi_n-i)}+\frac{\tilde{c}(dx_n)-i\tilde{c}(\xi')}{8(\xi_n-i)^2}
+\frac{3\xi_n-7i}{8(\xi_n-i)^3}[i\tilde{c}(\xi')-\tilde{c}(dx_n)]\right]\nonumber\\
&=&\frac{(3+i\xi_{n})h'(0)\epsilon(\xi')}{8b(\xi_{n}-i)^{3}} -\frac{(3+i\xi_{n})h'(0)\iota(\xi')}{8a(\xi_{n}-i)^{3}}
+\frac{(-i\xi_{n}^{2}-3\xi_{n}+4i)h'(0)\epsilon(dx_{n})}{8b(\xi_{n}-i)^{3}} \nonumber\\
&&-\frac{(-i\xi_{n}^{2}-3\xi_{n}+4i)h'(0)\iota(dx_{n})}{8a(\xi_{n}-i)^{3}}.
\end{eqnarray}
From (3.39) and (3.44), we obtain
\begin{eqnarray}
 &&\text{trace}\Big[B_{2}\times\partial_{\xi_{n}}\sigma_{-1}((\tilde{D^{*}})^{-1})\Big](x_{0})|_{|\xi'|=1}\nonumber\\
&=&\frac{h'(0)}{a^{2}}\frac{-i(4i-9\xi_{n}-7i\xi_{n}^{2}+3\xi_{n}^{2}+i\xi_{n}^{4})}{(\xi_{n}-i)^{3}(1+\xi_{n}^{2})^{2}}
+\frac{h'(0)}{b^{2}}\frac{-i(4i-9\xi_{n}-7i\xi_{n}^{2}+3\xi_{n}^{2}+i\xi_{n}^{4})}{(\xi_{n}-i)^{3}(1+\xi_{n}^{2})^{2}}.
\end{eqnarray}
Combining (3.43) and (3.45), we have
\begin{eqnarray}
{\rm case~ b)}&=&-i\int_{|\xi'|=1}\int_{-\infty}^{+\infty}\text{trace}\Big[(B_{1}-B_{2})
       \times\partial_{\xi_{n}}\sigma_{-1}((\tilde{D^{*}})^{-1})](x_{0})\texttt{d}\xi_{n}\sigma(\xi')\texttt{d}x' \nonumber\\
       &=&(- \frac{1}{8a^{2}}+\frac{11}{8b^{2}})\pi h'(0)  \Omega_{3}\texttt{d}x'.
\end{eqnarray}

\textbf{Case c}: \ $r=-1, \ \ell=-2, \ k=j=|\alpha|=0$

From(2.5) we have
\begin{equation}
\text{Case c}=-i\int_{|\xi'|=1}\int_{-\infty}^{+\infty}\text{trace}[\pi_{\xi_{n}}^{+}\sigma_{-1}(\tilde{D}^{-1})
 \times \partial_{\xi_{n}}\sigma_{-2}((\tilde{D^{*}})^{-1})](x_{0})\texttt{d}\xi_{n}\sigma(\xi')\texttt{d}x'  .
 \end{equation}
 By Lemma 3.1, Lemma 3.2  and Lemma 3.4, we obtain
\begin{eqnarray}
 \pi^+_{\xi_n}\sigma_{-1}(\tilde{D}^{-1})(x_0)|_{|\xi'|=1}&=&\frac{\tilde{c}(\xi')+i\tilde{c}(dx_n)}{2ab(\xi_n-i)} \nonumber\\
     &=& \frac{ \epsilon(\xi')}{2b(\xi_{n}-i) } -\frac{ \iota(\xi')}{2a(\xi_{n}-i) }
+\frac{i\epsilon(dx_{n})}{2b(\xi_{n}-i) } -\frac{i\iota(dx_{n})}{2a(\xi_{n}-i) }
\end{eqnarray}
and
\begin{eqnarray}
&&\partial_{\xi_{n}}\sigma_{-2}((\tilde{D^{*}})^{-1})](x_{0})\nonumber\\
&=&\frac{1}{a^{2}b^{2}(1+\xi_n^2)^3}\Big[(2\xi_n-2\xi_n^3)\bar{c}(dx_n)\bar{p_0}\bar{c}(dx_n)+(1-3\xi_n^2)\bar{c}(dx_n)\bar{p_0}\bar{c}(\xi') \nonumber\\
     &&+(1-3\xi_n^2)\bar{c}(\xi')\bar{p_0}\bar{c}(dx_n)-4\xi_n\bar{c}(\xi')\bar{p_0}\bar{c}(\xi')
+(3\xi_n^2-1)ab\partial_{x_n}\bar{c}(\xi')-4\xi_n\bar{c}(\xi')\bar{c}(dx_n)\partial_{x_n}\bar{c}(\xi') \nonumber\\
&&+2abh'(0)\bar{c}(\xi')+2abh'(0)\xi_n\bar{c}(dx_n)\Big]
+6\xi_nh'(0)\frac{\bar{c}(\xi)\bar{c}(dx_n)\bar{c}(\xi)}{a^{2}b^{2}(1+\xi^2_n)^4}\nonumber\\
&=&\frac{1}{a^{2}b^{2}(1+\xi_n^2)^3}\Big[(2\xi_n-2\xi_n^3)\bar{c}(dx_n)\bar{p_0}\bar{c}(dx_n)+(1-3\xi_n^2)\bar{c}(dx_n)\bar{p_0}\bar{c}(\xi') \nonumber\\
     &&+(1-3\xi_n^2)\bar{c}(\xi')\bar{p_0}\bar{c}(dx_n)-4\xi_n\bar{c}(\xi')\bar{p_0}\bar{c}(\xi') \Big] \nonumber\\
&&+\frac{(2-10\xi_{n}^{2})h'(0)\epsilon(\xi')}{a(1+\xi_{n}^{2})^{4}}
-\frac{(2-10\xi_{n}^{2})h'(0)\iota(\xi')}{b(1+\xi_{n}^{2})^{4}}
+\frac{(8\xi_{n}-4\xi_{n}^{3})h'(0)\epsilon(dx_{n})}{2a(1+\xi_{n}^{2})^{4}} \nonumber\\
&&-\frac{(8\xi_{n}-4\xi_{n}^{3})h'(0)\iota(dx_{n})}{b(1+\xi_{n}^{2})^{4}}
-\frac{3(\xi_{n}^{2}-1)\partial_{x_n}[\iota(\xi')](x_0)}{b(1+\xi_{n}^{2})^{3}}\nonumber\\
&&+\frac{4\xi_{n}\epsilon(\xi')\epsilon(dx_{n})\partial_{x_n}[\iota(\xi')](x_0)}{a(1+\xi_{n}^{2})^{3}}
 +\frac{4a\xi_{n}\iota(\xi')\iota(dx_{n})\partial_{x_n}[\iota(\xi')](x_0)}{b^{2}(1+\xi_{n}^{2})^{3}}.
\end{eqnarray}
Combining (3.48) and (3.49), we have
\begin{eqnarray}
 &&\text{trace}[\pi_{\xi_{n}}^{+}\sigma_{-1}(\tilde{D}^{-1})
 \times \partial_{\xi_{n}}\sigma_{-2}((\tilde{D^{*}})^{-1})](x_{0})|_{|\xi'|=1}\nonumber\\
&=&\frac{h'(0)}{a^{2}}\frac{-2(-7-21i\xi_{n}+26\xi_{n}^{2}+6i\xi_{n}^{3}+9\xi_{n}^{4}+3i\xi_{n}^{5})}{(\xi_{n}-i) (1+\xi_{n}^{2})^{4}}
+\frac{h'(0)}{b^{2}}\frac{-2(-1-25i\xi_{n}+26\xi_{n}^{2}+2i\xi_{n}^{3}+3\xi_{n}^{4}+3i\xi_{n}^{5})}{(\xi_{n}-i) (1+\xi_{n}^{2})^{4}}.\nonumber\\
\end{eqnarray}
Then similarly to computations of the case b), we have
\begin{equation}
{\rm case~ c)}= \frac{5}{2a^{2}} \pi h'(0)  \Omega_{3}\texttt{d}x'.
\end{equation}
Since $\Phi$  is the sum of the \textbf{case a, b} and \textbf{c},
\begin{equation}
\Phi=\frac{23}{8} (\frac{1}{a^{2}}+\frac{1}{b^{2}} )\pi h'(0)  \Omega_{3}\texttt{d}x'.
\end{equation}
Then we have
\begin{thm}\label{th:32}
Let M be a  four dimensional compact connected manifold with the boundary $\partial M$ and the metric $g^{M}$ as above ,
and $\tilde{D},\tilde{D^{*}} $ are the nonminimal de Rham-Hodge operators on $C^{\infty}({\Lambda^{*}(T^{*}M)})$, then
\begin{equation}
\widetilde{{\rm
Wres}}[\pi^+\tilde{D}^{-1} \circ \pi^+(\tilde{D^{*}})^{-1} ]
=4\pi\int_{M}\sum _{k=0} ^{4}c_{1}(4,k,a,b)R\texttt{d}vol(M)
-\frac{23}{12} (\frac{1}{a^{2}}+\frac{1}{b^{2}} )\pi\int_{\partial M} K  \Omega_{3}\texttt{d}x',
\end{equation}
where $R$ is the scalar curvature and
$c_{1}(4,k,a,b)=b^{-2}\{\frac{1}{6}(_{k}^{4})-(_{k-1}^{2})\}
+(b^{-2}-a^{-2})\sum_{j<k}(-1)^{j-k}\{\frac{1}{6}(_{j}^{4})-(_{j-1}^{2})\}.$
\end{thm}

Let us now consider the Einstein-Hilbert action for four dimensional manifolds with boundary.
Recall the Einstein-Hilbert action for manifolds with boundary\cite{Wa3},
\begin{equation}
I_{\rm Gr}=\frac{1}{16\pi}\int_MR{\rm dvol}_M+2\int_{\partial M}K{\rm dvol}_{\partial_M}:=I_{\rm {Gr,i}}+I_{\rm {Gr,b}},
\end{equation}
 where
 \begin{equation}
K=\sum_{1\leq i,j\leq {n-1}}K_{i,j}g_{\partial M}^{i,j};~~K_{i,j}=-\Gamma^n_{i,j},
\end{equation}
and $K_{i,j}$ is the second fundamental form, or extrinsic curvature. Take the metric in Section 2,
$K_{i,j}(x_0)=-\Gamma^n_{i,j}(x_0)=-\frac{1}{2}h'(0),$ when $i=j<n$,
otherwise is zero.

Let
 \begin{equation}
\widetilde{{\rm Wres}}[\pi^+(\tilde{D} )^{-1}\circ\pi^+(\tilde{D ^{*}})^{-1}]
=\widetilde{{\rm Wres}}_{i}[\pi^+\tilde{D }^{-1}\circ\pi^+(\tilde{D ^{*}})^{-1}]+\widetilde{{\rm Wres}}_{b}[\pi^+\tilde{D }^{-1}\circ\pi^+(\tilde{D ^{*}})^{-1}],
\end{equation}
where
 \begin{equation}
\widetilde{{\rm Wres}}_{i}[\pi^+(\tilde{D} )^{-1}\circ\pi^+(\tilde{D ^{*}})^{-1}]
=\int_{M}\int_{|\xi|=1}
  \text{trace}_{\Lambda^{*}(T^{*}M)}[\sigma_{-4}((\tilde{D}\tilde{D^{*}})^{-1})]\sigma(\xi)\text{d}x
\end{equation}
and
\begin{eqnarray}
&&\widetilde{{\rm Wres}}_{b}[\pi^+(\tilde{D} )^{-1}\circ\pi^+(\tilde{D ^{*}})^{-1}]\nonumber\\
&=& \int_{\partial M}\int_{|\xi'|=1}\int^{+\infty}_{-\infty}\sum^{\infty}_{j, k=0}\sum\frac{(-i)^{|\alpha|+j+k+1}}{\alpha!(j+k+1)!}
\times {\rm trace}_{\Lambda^{*}(T^{*}M)}[\partial^j_{x_n}\partial^\alpha_{\xi'}\partial^k_{\xi_n}\sigma^+_{r}((\tilde{D}^{-1})(x',0,\xi',\xi_n)
\nonumber\\
&&\times\partial^\alpha_{x'}\partial^{j+1}_{\xi_n}\partial^k_{x_n}\sigma_{l}((\tilde{D ^{*}})^{-1})(x',0,\xi',\xi_n)]d\xi_n\sigma(\xi')dx'
\end{eqnarray}
denote the interior term and boundary term of $\widetilde{{\rm Wres}}[\pi^+(\tilde{D} )^{-1}\circ\pi^+(\tilde{D ^{*}})^{-1}]$.
Combining (3.42), (3.43) and (3.45), we obtain
\begin{thm}
 Let M be a  four dimensional compact manifold  with the boundary $\partial M$ associated to nonminimal de Rham-Hodge operators $\tilde{D}$ and
 $\tilde{D ^{*}}$. Assume $\partial M$ is flat, then
\begin{eqnarray}
&& I_{\rm {Gr,i}}=\frac{1}{64\pi c_{1}(4,k,a,b)} \widetilde{{\rm Wres}}_{i}[\pi^+(\tilde{D} )^{-1}\circ\pi^+(\tilde{D ^{*}})^{-1}]; \nonumber\\
&&I_{\rm {Gr,b}}=\frac{-24}{23(\frac{1}{a^{2}}+\frac{1}{b^{2}} ) \pi\Omega_3 }\widetilde{{\rm Wres}}_{b}[\pi^+(\tilde{D} )^{-1}\circ\pi^+(\tilde{D ^{*}})^{-1}].
\end{eqnarray}
\end{thm}

\section{The Kastler-Kalau-Walze type theorem of the nonminimal de Rham-Hodge operators  $\tilde{D}$ }
 \label{4}
In this section, we compute the lower dimension volume for four dimension compact connected manifolds with boundary
associated to nonminimal de Rham-Hodge operators $\tilde{D}$ and get a Kastler-Kalau-Walze type theorem in this case.
Let $M$ be an four dimensional compact oriented connected manifold with boundary $\partial M$, and the metric $g^{M}$ on $M$ as above.
Note that $[\sigma_{-4}(\tilde{D}^{-2})]|_{M}$ has the same expression with the case of  without boundary in \cite{Wa3},
so locally we can use Theorem 3.1 in \cite{Wa3} to compute the first term. Therefore
\begin{equation}
\int_{M}\int_{|\xi|=1}
  \text{trace}_{{\Lambda^{*}(T^{*}M)}}[\sigma_{-4}(\tilde{D}^{-2})]\sigma(\xi)\text{d}x
=\frac{8\Omega_4}{3ab}\int_MR{\rm
dvol}_M,
\end{equation}
where $R$ is the scalar curvature.

Let us now turn to compute $\Phi$ (see formula (2.5) for definition of $\Phi$). Since the sum is taken over $-r-\ell+k+j+|\alpha|=3,
 \ r, \ell\leq-1$, then we have the following five cases:

\textbf{Case a (I)}: \ $r=-1, \ \ell=-1, \ k=j=0, \ |\alpha|=1$

From (2.5) we have
\begin{equation}
\text{ Case a (\text{I}) }=-\int_{|\xi'|=1}\int_{-\infty}^{+\infty}\sum_{|\alpha|=1}\text{trace}
   \Big[\partial_{\xi'}^{\alpha}\pi_{\xi_{n}}^{+}\sigma_{-1}(\tilde{D}^{-1})
\times \partial_{x'}^{\alpha}\partial_{\xi_{n}}\sigma_{-1}(\tilde{D}^{-1})\Big](x_{0})\texttt{d}\xi_{n}
\sigma(\xi')\texttt{d}x'.
\end{equation}
Then an application of Lemma 3.2 shows that,
\begin{equation}
\partial_{x_i}\sigma_{-1}(\tilde{D}^{-1})(x_0)=\partial_{x_i}\left(\frac{\sqrt{-1}\tilde{c}(\xi)}{ab|\xi|^2}\right)(x_0)=
\frac{\sqrt{-1}\partial_{x_i}[\tilde{c}(\xi)](x_0)}{ab|\xi|^2}
-\frac{\sqrt{-1}\tilde{c}(\xi)\partial_{x_i}(|\xi|^2)(x_0)}{ab|\xi|^4}=0,
 \end{equation}
so Case a (I) vanishes.

\textbf{Case a (II)}: \ $r=-1, \ \ell=-1, \ k=|\alpha|=0, \ j=1$

From (2.5) we have
\begin{equation}
\text{ Case a (\text{II}) }=-\frac{1}{2}\int_{|\xi'|=1}\int_{-\infty}^{+\infty}
             \text{trace}[\partial_{x_{n}}\pi_{\xi_{n}}^{+}\sigma_{-1}(\tilde{D}^{-1})
\times \partial_{\xi_{n}}^{2}\sigma_{-1}(\tilde{D}^{-1})](x_{0})\texttt{d}\xi_{n}\sigma(\xi')\texttt{d}x'.
\end{equation}
By  Lemma 3.1, Lemma 3.2,  we have
\begin{equation}
\partial_{x_{n}}\sigma_{-1}(\tilde{D}^{-1})(x_{0})|_{|\xi'|=1}=\frac{\sqrt{-1}\partial_{x_n}[\tilde{c}(\xi)](x_0)}{ab|\xi|^2}
-\frac{\sqrt{-1}\tilde{c}(\xi)h'(0)}{ab|\xi|^4}.
\end{equation}
 By the Cauchy integral formula we obtain
\begin{equation}
\pi_{\xi_{n}}^{+}[\frac{1}{(1+\xi_{n}^{2})^{2}}](x_{0})|_{|\xi'|=1}
  =\frac{1}{2\pi i}\lim_{u\rightarrow 0^{-}}\int_{\Gamma^{+}}\frac{\frac{1}{(\eta_{n}+i)^{2}(\xi_{n}+iu-\eta_{n})}}
{(\eta_{n}-i)^{2}}\texttt{d}\eta_{n}=-\frac{i\xi_{n}+2}{4(\xi_{n}-i)^{2}},
\end{equation}
and
\begin{equation}
\pi^+_{\xi_n}\left[\frac{\sqrt{-1}\partial_{x_n}c(\xi')}{|\xi|^2}\right](x_0)|_{|\xi'|=1}
=\frac{\partial_{x_n}[c(\xi')](x_0)}{2(\xi_n-i)}.
\end{equation}
Then
\begin{eqnarray}
&&\partial_{x_{n}}\pi_{\xi_{n}}^{+}\sigma_{-1}(\tilde{D}^{-1})(x_{0})|_{|\xi'|=1}\nonumber\\
&=&\frac{\partial_{x_n}[\tilde{c}(\xi')](x_0)}{2ab(\xi_{n}-i)}
+\frac{ih'(0)}{ab}\Big[ \frac{i\tilde{c}(\xi')}{4(\xi_{n}-i)} +\frac{\tilde{c}(\xi')+i\tilde{c}(dx_{n})}{4(\xi_{n}-i)^{2}}\Big].
\nonumber\\
&=&-\frac{\partial_{x_n}[\iota(\xi')](x_0)}{2a(\xi_{n}-i)}
+\frac{ih'(0)}{ab}\Big[ \frac{a(i\xi_{n}+2)\epsilon(\xi')}{4(\xi_{n}-i)^{2}} -\frac{b(i\xi_{n}+2)\iota(\xi')}{4(\xi_{n}-i)^{2}}
+\frac{ai\epsilon(dx_{n})}{4(\xi_{n}-i)^{2}} -\frac{bi\iota(dx_{n})}{4(\xi_{n}-i)^{2}}\Big].
\end{eqnarray}
By Lemma 3.1 and Lemma 3.2, we have
\begin{eqnarray}
 &&\partial_{\xi_{n}}^{2}\sigma_{-1}(\tilde{D}^{-1})(x_{0})|_{|\xi'|=1}\nonumber\\
&=&\frac{\sqrt{-1}}{ab}\left(-\frac{6\xi_n\tilde{c} (dx_n)+2\tilde{c}(\xi')}{|\xi|^4}+\frac{8\xi_n^2\tilde{c} (\xi)}{|\xi|^6}\right)\nonumber\\
&=&\frac{i(6\xi_{n}^{2}-2)\epsilon(\xi')}{b(1+\xi_{n}^{2})^{3}} -\frac{i(6\xi_{n}^{2}-2)\iota(\xi')}{a(1+\xi_{n}^{2})^{3}}
+\frac{i(2\xi_{n}^{3}-6\xi_{n})\epsilon(dx_{n})}{b(1+\xi_{n}^{2})^{3}}-\frac{i(2\xi_{n}^{3}-6\xi_{n})\iota(dx_{n})}{a(1+\xi_{n}^{2})^{3}}.
\end{eqnarray}
Combining (4.5) and (4.9), we have
\begin{equation}
\text{trace}[\partial_{x_{n}}\pi_{\xi_{n}}^{+}\sigma_{-1}(\tilde{D}^{-1})
\times \partial_{\xi_{n}}^{2}\sigma_{-1}(\tilde{D}^{-1})](x_{0})|_{|\xi'|=1}
=\frac{h'(0)}{ab}\frac{8(-1-3i\xi_{n}+3\xi_{n}^{2}+i\xi_{n}^{3})}{(\xi_{n}-i)^{2}(1+\xi_{n}^{2})^{3}}.
\end{equation}
Hence
\begin{eqnarray}
\text{ Case a (\text{II})}
&=&-\frac{1}{2}\int_{|\xi'|=1}\int_{-\infty}^{+\infty}
 \frac{h'(0)}{ab}\frac{8(-1-3i\xi_{n}+3\xi_{n}^{2}+i\xi_{n}^{3})}{(\xi_{n}-i)^{2}(1+\xi_{n}^{2})^{3}}\texttt{d}\xi_{n}\sigma(\xi')\texttt{d}x'\nonumber\\
&=&-\frac{1}{2} \frac{h'(0)}{ab}\Omega_{3}\int_{\Gamma^{+}}\frac{8(-1-3i\xi_{n}+3\xi_{n}^{2}+i\xi_{n}^{3})}{(\xi_{n}-i)^{2}(1+\xi_{n}^{2})^{3}}
\texttt{d}\xi_{n}\texttt{d}x'
\nonumber\\
&=&-\frac{1}{2} \frac{h'(0)}{a^{2}}\Omega_{3}\frac{2\pi i}{4!}
\Big[\frac{ 8(-1-3i\xi_{n}+3\xi_{n}^{2}+i\xi_{n}^{3})}{(\xi_{n}+i)^{3}}\Big]^{(4)}|_{\xi_{n}=i}\texttt{d}x'
 \nonumber\\
&=&-\frac{1}{2} \frac{h'(0)}{a^{2}}\Omega_{3}\frac{2\pi i}{4!}(-36i)\texttt{d}x'
 \nonumber\\
&=&-\frac{3}{2ab}\pi h'(0)  \Omega_{3}\texttt{d}x'.
\end{eqnarray}

\textbf{Case a (III)}: \ $r=-1, \ \ell=-1, \ j=|\alpha|=0, \ k=1$

 From (2.5) we have
 \begin{equation}
\text{ Case a (\text{III}) }=-\frac{1}{2}\int_{|\xi'|=1}\int_{-\infty}^{+\infty}\text{trace}\Big[\partial_{\xi_{n}}\pi_{\xi_{n}}^{+}
\sigma_{-1}(\tilde{D}^{-1})\times\partial_{\xi_{n}}\partial_{x_{n}}\sigma_{-1}(\tilde{D}^{-1})\Big](x_{0})
    \texttt{d}\xi_{n}\sigma(\xi')\texttt{d}x'.
 \end{equation}

From Lemma 3.1 and Lemma 3.2 we obtain
\begin{eqnarray}
 \partial_{\xi_n}\pi^+_{\xi_n}\sigma_{-1}(\tilde{D}^{-1})(x_0)|_{|\xi'|=1}  &=&-\frac{\tilde{c}(\xi')+i\tilde{c}(dx_n)}{2ab(\xi_n-i)^2} \nonumber\\
      &=&-\frac{a\epsilon(\xi')-b\iota(\xi') +i(\epsilon(dx_n)+\iota(dx_n))}{2ab(\xi_n-i)^2},
\end{eqnarray}
and
\begin{eqnarray}
&&\partial_{\xi_n}\partial_{x_n}\sigma_{-1}(\tilde{D}^{-1}(x_0)|_{|\xi'|=1}\nonumber\\
&=&\frac{-\sqrt{-1}h'(0)}{ab}
\left[\frac{\tilde{c} (dx_n)}{|\xi|^4}-4\xi_n\frac{\tilde{c} (\xi')+\xi_n\tilde{c} (dx_n)}{|\xi|^6}\right]-
\frac{2\xi_n\sqrt{-1}\partial_{x_n}\tilde{c} (\xi')(x_0)}{ab|\xi|^4}\nonumber\\
&=&\frac{2i\xi_{n}\partial_{x_n}[\iota(\xi')](x_0)}{a(1+\xi_{n}^{2})^{2}}
+\frac{4 i\xi_{n}h'(0)\epsilon(\xi')}{b(1+\xi_{n}^{2})^{3}} -\frac{4i\xi_{n}h'(0)\iota(\xi')}{a(1+\xi_{n}^{2})^{3}}
+\frac{( 3\xi_{n}^{2}-1)i h'(0)\epsilon(dx_{n})}{b(1+\xi_{n}^{2})^{3}}
-\frac{( 3\xi_{n}^{2}-1)i h'(0)\iota(dx_{n})}{a(1+\xi_{n}^{2})^{3}}.\nonumber\\
\end{eqnarray}
From (4.13) and (4.14), we have
\begin{equation}
\text{trace}\Big[\partial_{\xi_{n}}\pi_{\xi_{n}}^{+}
\sigma_{-1}(\tilde{D}^{-1})\times\partial_{\xi_{n}}\partial_{x_{n}}\sigma_{-1}(\tilde{D}^{-1})\Big](x_{0})|_{|\xi'|=1}
=\frac{h'(0)}{ab}\frac{8i(-i+3\xi_{n}+3i\xi_{n}^{2}-\xi_{n}^{3})}{(\xi_{n}-i)^{2}(1+\xi_{n}^{2})^{3}}.
\end{equation}
Then
\begin{eqnarray}
\text{ Case a (\text{III}) }
&=&-\frac{1}{2}\int_{|\xi'|=1}\int_{-\infty}^{+\infty}
 \frac{h'(0)}{ab}\frac{8i(-i+3\xi_{n}+3i\xi_{n}^{2}-\xi_{n}^{3})}{(\xi_{n}-i)^{2}(1+\xi_{n}^{2})^{3}}\texttt{d}\xi_{n}\sigma(\xi')\texttt{d}x'\nonumber\\
\nonumber\\
&=&-\frac{1}{2} \frac{h'(0)}{ab}\Omega_{3}\int_{\Gamma^{+}}\frac{8i(-i+3\xi_{n}+3i\xi_{n}^{2}-\xi_{n}^{3})}{(\xi_{n}-i)^{2}(1+\xi_{n}^{2})^{3}}
\texttt{d}\xi_{n}\texttt{d}x'\nonumber\\
&=&-\frac{1}{2} \frac{h'(0)}{a^{2}}\Omega_{3}\frac{2\pi i}{4!}(36i)\texttt{d}x'\nonumber\\
&=&\frac{3}{2ab}\pi h'(0)  \Omega_{3}\texttt{d}x'.
\end{eqnarray}

\textbf{Case b}: \ $r=-1, \ \ell=-2, \ k=j=|\alpha|=0$

From (2.5) and the Leibniz rule, we obtain
\begin{eqnarray}
\text{Case b}&=&-i\int_{|\xi'|=1}\int_{-\infty}^{+\infty}\text{trace}[\pi_{\xi_{n}}^{+}\sigma_{-1}(\tilde{D}^{-1})
 \times \partial_{\xi_{n}}\sigma_{-2}(\tilde{D}^{-1})](x_{0})\texttt{d}\xi_{n}\sigma(\xi')\texttt{d}x'    \nonumber\\
  &=&i\int_{|\xi'|=1}\int_{-\infty}^{+\infty}\text{trace}[\partial_{\xi_{n}}\pi_{\xi_{n}}^{+}\sigma_{-1}(\tilde{D}^{-1})
 \times \sigma_{-2}(\tilde{D}^{-1})](x_{0})\texttt{d}\xi_{n}\sigma(\xi')\texttt{d}x'
\end{eqnarray}
From Lemma 3.1, Lemma 3.2 and Lemma 3.4 we obtain
\begin{eqnarray}
&&\sigma_{-2}(\tilde{D}^{-1})(x_{0})\nonumber\\
&=&\frac{\tilde{c}(\xi)\tilde{p}_0(x_0)\tilde{c}(\xi)}{a^{2}b^{2}|\xi|^4}
+\frac{\tilde{c}(\xi)}{a^{2}b^{2}|\xi|^6}\tilde{c}(dx_n)
\Big[\partial_{x_n}[\tilde{c}(\xi')](x_0)|\xi|^2-\tilde{c}(\xi)h'(0)|\xi|^2_{\partial M}\Big]\nonumber\\
&=&\frac{1}{a^{2}b^{2}(1+\xi_n^{2})^2}\Big[\tilde{c}(\xi')\tilde{p}_0\tilde{c}(\xi')+\xi_n\tilde{c}(dx_n)\tilde{p_0}\tilde{c}(\xi')
+\xi_{n}\tilde{c}(\xi')\tilde{p_0}\tilde{c}(dx_{n})+\xi_n^{2}\tilde{c}(dx_n)\tilde{p}_0\tilde{c}(dx_n)\Big]\nonumber\\
&&+\frac{2\xi_{n}h'(0)\epsilon(\xi')}{b(1+\xi_{n}^{2})^{3}}
-\frac{2\xi_{n}h'(0)\iota(\xi')}{a(1+\xi_{n}^{2})^{3}}
+\frac{(\xi_{n}^{2}-1)h'(0)\epsilon(dx_{n})}{b(1+\xi_{n}^{2})^{3}} \nonumber\\
&&+\frac{(1-\xi_{n}^{2})h'(0)\iota(dx_{n})}{a(1+\xi_{n}^{2})^{3}}
+\frac{\xi_{n} \partial_{x_n}[\iota(\xi')](x_0)}{a(1+\xi_{n}^{2})^{2}}\nonumber\\
&&+\frac{ \epsilon(\xi')\epsilon(dx_{n})\partial_{x_n}[\iota(\xi')](x_0)}{b(1+\xi_{n}^{2})^{2}}
-\frac{b\xi_{n}\iota(\xi')\iota(dx_{n})\partial_{x_n}[\iota(\xi')](x_0)}{a^{2}(1+\xi_{n}^{2})^{2}}.
\end{eqnarray}
From (4.13) and (4.18), we obtain
\begin{equation}
\text{trace}[\partial_{\xi_{n}}\pi_{\xi_{n}}^{+}\sigma_{-1}(\tilde{D}^{-1})
 \times \sigma_{-2}(\tilde{D}^{-1})](x_{0})|_{|\xi'|=1}\nonumber\\
=\frac{h'(0)}{ab}\frac{2(-8i+12\xi_{n}+3i\xi_{n}^{2}+4\xi_{n}^{3}+3i\xi_{n}^{4})}{(\xi_{n}-i)^{2} (1+\xi_{n}^{2})^{3}}.
\end{equation}
Then similarly to computations of the case a), we have
\begin{equation}
{\rm case~ b)}=- \frac{6}{ab} \pi h'(0)  \Omega_{3}\texttt{d}x'.
\end{equation}

\textbf{Case c}: \ $r=-2, \ \ell=-1, \ k=j=|\alpha|=0$

From (2.5) we have
\begin{equation}
\text{ Case c}=-i\int_{|\xi'|=1}\int_{-\infty}^{+\infty}\text{trace}[\pi_{\xi_{n}}^{+}\sigma_{-2}(\tilde{D}^{-1})
       \times\partial_{\xi_{n}}\sigma_{-1}(\tilde{D}^{-1})](x_{0})\texttt{d}\xi_{n}\sigma(\xi')\texttt{d}x' .
\end{equation}
By the Leibniz rule, trace property and "++" and "-~-" vanishing
after the integration over $\xi_n$\cite{FGLS}\cite{GS}, then
\begin{eqnarray*}
& &\int^{+\infty}_{-\infty}{\rm trace}
[\pi^+_{\xi_n}\sigma_{-2}(\tilde{D}^{-1})\times
\partial_{\xi_n}\sigma_{-2}(\tilde{D}^{-2})]d\xi_n\\
&=&\int^{+\infty}_{-\infty}{\rm tr} [\sigma_{-2}(D^{-2})\times
\partial_{\xi_n}\sigma_{-1}(\tilde{D}^{-1})]d\xi_n
-\int^{+\infty}_{-\infty}{\rm tr}
[\pi^-_{\xi_n}\sigma_{-2}(\tilde{D}^{-1})\times
\partial_{\xi_n}\sigma_{-1}(\tilde{D}^{-1})]d\xi_n\\
&=&\int^{+\infty}_{-\infty}{\rm tr} [\sigma_{-2}(D^{-1})\times
\partial_{\xi_n}\sigma_{-1}(\tilde{D}^{-2})]d\xi_n-\int^{+\infty}_{-\infty}{\rm
tr} [\pi^-_{\xi_n}\sigma_{-2}(\tilde{D}^{-2})\times
\partial_{\xi_n}\pi^+_{\xi_n}\sigma_{-1}(\tilde{D}^{-1})]d\xi_n\\
&=&\int^{+\infty}_{-\infty}{\rm tr} [\sigma_{-2}(\tilde{D}^{-1})\times
\partial_{\xi_n}\sigma_{-1}(\tilde{D}^{-1})]d\xi_n-\int^{+\infty}_{-\infty}{\rm
tr} [\sigma_{-2}(\tilde{D}^{-1})\times
\partial_{\xi_n}\pi^+_{\xi_n}\sigma_{-1}(\tilde{D}^{-1})]d\xi_n\\
&=&\int^{+\infty}_{-\infty}{\rm tr} [
\partial_{\xi_n}\sigma_{-1}(\tilde{D}^{-1})\times\sigma_{-2}(\tilde{D}^{-1})]d\xi_n+\int^{+\infty}_{-\infty}{\rm
tr}[\partial_{\xi_n}\sigma_{-2}(\tilde{D}^{-2})\times
\pi^+_{\xi_n}\sigma_{-1}(\tilde{D}^{-1})]d\xi_n\\
&=&\int^{+\infty}_{-\infty}{\rm tr} [
\partial_{\xi_n}\sigma_{-1}(\tilde{D}^{-1})\times\sigma_{-2}(\tilde{D}^{-2})]d\xi_n+\int^{+\infty}_{-\infty}{\rm
tr}[\pi^+_{\xi_n}\sigma_{-1}(\tilde{D}^{-1})\times
\partial_{\xi_n}\sigma_{-2}(\tilde{D}^{-1})]d\xi_n.
\end{eqnarray*}
Then we have
\begin{equation}
{\rm {\bf case~ c)}}={\rm {\bf case~ b)}}-i\int_{|\xi'|=1}\int^{+\infty}_{-\infty}
{\rm tr}[\partial_{\xi_n}\sigma_{-1}(\tilde{D}^{-1})\times \sigma_{-2}(\tilde{D}^{-1})]d\xi_n\sigma(\xi')dx'.
\end{equation}
By  Lemma 3.1, Lemma 3.2, we obtain
  \begin{eqnarray}
&&\partial_{\xi_{n}}\sigma_{-1}(\tilde{D}^{-1})](x_{0})\nonumber\\
&=&\frac{\sqrt{-1}}{ab}\left(-\frac{2\xi_n^{2}\tilde{c}(dx_n)+2\xi_n\tilde{c}(\xi')}
{|\xi|^4}+\frac{\tilde{c}(dx_n)}{|\xi|^2}\right)\nonumber\\
&=& \frac{-2i\xi_{n}\epsilon(\xi')}{b(1+\xi_{n}^{2})^{2}} +\frac{2 i \xi_{n}\iota(\xi')}{2(1+\xi_{n}^{2})^{2}}
+\frac{i(1-\xi_{n}^{2})\epsilon(dx_{n})}{b(1+\xi_{n}^{2})^{2}} +\frac{i(\xi_{n}^{2}-1)\iota(dx_{n})}{a(1+\xi_{n}^{2})^{2}}.
\end{eqnarray}
 From (4.18) and (4.22), we have
\begin{equation}
-i\int_{|\xi'|=1}\int^{+\infty}_{-\infty}
{\rm tr}[\partial_{\xi_n}\sigma_{-1}(\tilde{D}^{-1})\times \sigma_{-2}(\tilde{D}^{-1})]d\xi_n\sigma(\xi')dx'
= \frac{12}{ab} \pi h'(0)  \Omega_{3}\texttt{d}x'.
\end{equation}
By (4.19), (4.21) and  (4.23), we have
\begin{equation}
{\rm case~ c)}=\frac{6}{ab} \pi h'(0)  \Omega_{3}\texttt{d}x'.
\end{equation}
Since $\Phi$  is the sum of the \textbf{case a, b} and \textbf{c}, so is zero.
Then we have
\begin{thm}\label{th:32}
Let M be a four dimensional compact connected manifold with the boundary $\partial M$ and the metric $g^{M}$ as above ,
and $\tilde{D}$ are the nonminimal de Rham-Hodge operators on $C^{\infty}({\Lambda^{*}(T^{*}M)})$, then
\begin{equation}
\widetilde{{\rm
Wres}}[(\pi^+\tilde{D}^{-1})^2]=\frac{8\Omega_4}{3ab}\int_MR{\rm dvol}_M,
\end{equation}
where $R$ is the scalar curvature and  $\Omega_{4}$ is the canonical volume of $S^{3}$.
\end{thm}

\section{Lower dimensional volumes for three dimensional spin manifolds with boundary }

For an odd dimensional manifolds with boundary, as in Section
5-7 in \cite{Wa1}, we have the formula
\begin{equation}
\widetilde{{\rm Wres}}[\pi^+\tilde{D}^{-1} \circ \pi^+(\tilde{D^{*}})^{-1} ]=\int_{\partial M}\Phi.
\end{equation}
When $n=3$, then in (2.5), $ r-k-|\alpha|+l-j-1=-3,~~r,l\leq-1$, so
we get $r=l=-1,~k=|\alpha|=j=0,$
\begin{equation}
\Phi=\int_{|\xi'|=1}\int^{+\infty}_{-\infty}  {\rm trace}_{S(TM)}
[ \sigma^+_{-1}(\tilde{D}^{-1})(x',0,\xi',\xi_n)\times\partial_{\xi_n}\sigma_{-1}((\tilde{D^{*}})^{-1})(x',0,\xi',\xi_n)]d\xi_3\sigma(\xi')dx'.
\end{equation}
 By Lemma 3.1 and Lemma 3.2, we have
\begin{eqnarray}
\sigma^+_{-1}(\tilde{D}^{-1})|_{|\xi'|=1}(x_0)|_{|\xi'|=1}  &=&\frac{\tilde{c}(\xi')+i\tilde{c}(dx_n)}{2ab(\xi_n-i)} \nonumber\\
      &=&\frac{a\epsilon(\xi')-b\iota(\xi') +i(a\epsilon(dx_n)-b\iota(dx_n))}{2ab(\xi_n-i)},
\end{eqnarray}
and
\begin{eqnarray}
\partial_{\xi_n}\sigma_{-1}((\tilde{D^{*}})^{-1})(x_0)|_{|\xi'|=1}  &=&\frac{\sqrt{-1}\bar{c}(dx_n)}{ab(1+\xi_n^2)}
-\frac{2\sqrt{-1}\xi_n\bar{c}(\xi)}{ab(1+\xi_n^2)^2}\nonumber\\
      &=&\frac{\sqrt{-1}(b\epsilon(dx_n)-a\iota(dx_n))}{ab(1+\xi_n^2)}-\frac{2\sqrt{-1}\xi_n(b\epsilon(\xi)-a\iota(\xi)) }{ab(1+\xi_n^2)^2}.
\end{eqnarray}
For $n=3$, we take the coordinates as in Section 2.  Locally $S(TM)|_{\widetilde {U}}\cong
\widetilde {U}\times\wedge^{{\rm even}} _{\bf C}(2).$ Let $\{\widetilde{f_1},\widetilde{f_2}\}$ be an orthonormal basis of
$\wedge^{{\rm even}} _{\bf C}(2)$ and we will compute the trace under this basis.
By the relation of the Clifford action and ${\rm tr}{AB}={\rm tr }{BA}$, then we have the equalities:
\begin{equation}
{\rm tr}[\epsilon(\xi')\iota(\xi')]=4;~~{\rm tr}[\epsilon(dx_n)\iota(dx_n)]=4.
\end{equation}
Form (5.3) (5.4) and (5.5), we get
\begin{equation}
{\rm trace} [\sigma^+_{-1}(\tilde{D}^{-1})\times\partial_{\xi_n}\sigma_{-1}((\tilde{D^{*}})^{-1})](x_0)|_{|\xi'|=1}
=\frac{1}{a^{2}}\frac{ -2\xi_{n}^{3}+2\xi_{n}+4i\xi_{n} }{(\xi_{n}-i)(1+\xi_{n}^{2})^{2}}
+\frac{1}{b^{2}}\frac{ -2\xi_{n}^{3}+2\xi_{n}+4i\xi_{n} }{(\xi_{n}-i)(1+\xi_{n}^{2})^{2}}.
\end{equation}
 By (5.2) and (5.6)
and the Cauchy integral formula, we get
\begin{equation}
\Phi=(-\frac{1}{2}+\frac{i}{2})(\frac{1}{a^{2}}+\frac{1}{b^{2}})\pi\Omega_2{\rm vol}_{\partial M},
\end{equation}
where ${\rm vol}_{\partial M}$ denotes the canonical volume form of ${\partial M}$. Then we obtain
\begin{thm}
Let $M$ be a three dimensional
compact spin manifold with the boundary $\partial M$ and the metric
$g^M$ as in Section 2, and $\tilde{D}$ be the nonminimal de Rham-Hodge operators  on $\widehat{M}$, then
\begin{equation}
\widetilde{{\rm Wres}}[\pi^+\tilde{D}^{-1} \circ \pi^+(\tilde{D^{*}})^{-1} ]=(-\frac{1}{2}+\frac{i}{2})(\frac{1}{a^{2}}+\frac{1}{b^{2}})\pi\Omega_2{\rm vol}_{\partial M},
 \end{equation}
where ${\rm vol}_{\partial M}$ denotes the canonical volume of ${\partial M}.$
\end{thm}

Authors: Jian Wang,wangj484@nenu.edu.cn, Yong Wang,wangy581@nenu.edu.cn
 and Aihui Sun, Sihui Chen,
(Addresses: School of Mathematics and Statistics, Northeast Normal University, Changchun, 130024, P.R.China)

\end{document}